\documentclass[runningheads]{llncs}

\usepackage{color}
\definecolor{blue}{rgb}{0,0,.7}

\definecolor{red}{rgb}{1,0,0}

\usepackage{makeidx}  
\usepackage{epsfig}
\usepackage{amsmath}

\usepackage{graphicx}
\usepackage{epstopdf}
\usepackage{url,
epsfig,
amssymb,
}

\usepackage{xcolor}

\begin{document}
%
%
\pagestyle{headings}  
\addtocmark{} 
%
\mainmatter              
\title{Graphs with degree sequence \\$\{m^{m-1},n^{n-1}\}$ and $\{m^n,n^m\}$}
\titlerunning{Graphs with degree sequence \\$\{m^{m-1},n^{n-1}\}$ and $\{m^n,n^m\}$}
%
\author{Boris Brimkov$^\text{1}$, Valentin Brimkov$^\text{2}$}

\authorrunning{B. Brimkov, V. Brimkov}

\institute{
$^\text{1}$ Department of Mathematics and Statistics,\\ Slippery Rock University, Slippery Rock, PA, USA\\
\email{boris.brimkov@sru.edu}\\
\medskip
$^\text{2}$ Mathematics Department,\\ 
SUNY Buffalo State University, Buffalo, NY, USA\\
\email{brimkove@buffalostate.edu}
}

\maketitle              

\begin{abstract}
In this paper we study the class of graphs $G_{m,n}$  that have the same degree sequence as two disjoint cliques $K_m$ and $K_n$, as well as the class $\overline G_{m,n}$ of the complements of such graphs. 
We establish various properties of $G_{m,n}$ and  $\overline G_{m,n}$ related to recognition, connectivity, 
diameter,
bipartiteness,
Hamiltonicity, and pancyclicity. We also show that several classical optimization problems on these graphs are NP-hard. 
\smallskip

\end{abstract}

\section{Introduction}
\label{intro}

Two graphs $G_1$ and $G_2$ are called {\em degree-equivalent} if they have the same degree sequence. 
In this paper we present results about simple graphs that are different from and degree-equivalent to a disjoint union of cliques $K_{p_1},\ldots,K_{p_k}$, as well as the complements of such graphs.
We will denote this class of graphs by $G_{p_1,\dots, p_k}$ and the class of their complements by $\overline G_{p_1,\dots, p_k}$. Predominantly, we will study graphs that are degree-equivalent to two disjoint cliques and their complements, i.e., the graphs in $G_{m,n}$ and $\overline G_{m,n}$. Graphs in these families have degree sequence $\{m^{m-1},n^{n-1}\}$ and $\{m^n,n^m\}$, respectively, where $a^x$ means the number $a$ appears $x$ times in the sequence.

The graphs in $G_{m,n}$ and $\overline G_{m,n}$ can model peer-to-peer file sharing networks where, for example, each of $n$ downloaders must be connected to $m$ peers and each of $m$ uploaders must be connected to $n$ peers. They can also model networks which start out as disjoint cliques or as complete bipartite graphs, and then evolve through a sequence of 2-switches (see Section \ref{prelim} for definitions). We identify properties that are shared by all networks in such a sequence of transformations.
In particular, we show that graphs in $G_{m,n}$ and $\overline{G}_{m,n}$ can have a rather varied structure, yet possess desirable qualities such as Hamiltonicity, traceability, bounded diameter, and efficient recognizability. 
On the other hand, various optimization problems such as maximum clique, maximum independent set, and minimum vertex cover are NP-hard on these classes of graphs. 

The graphs we examine in this paper are somewhat reminiscent of {\em biregular}, {\em semiregular}, {\em almost regular}, and {\em nearly regular} graphs which have been studied in a number of works \cite{albertson,alon,joentgen,macon,nelson,nelson2}. Usually, an \emph{almost regular} graph is defined as a graph whose vertex degrees differ by at most one. 
Some variations or generalizations of this definition have been considered as well~\cite{erdos,erdos2}. 
A special example of a graph that belongs to the class $\overline G_{m,n}$ is the well-known {\em Mantel graph} \cite{mantel}; it is a graph on $n$ vertices, which is the complement of a graph that consists of two cliques of size $\lfloor n/2 \rfloor$ and $\lceil n/2 \rceil$. The graphs of two of the Platonic solids, the cube and the icosahedron, are in $G_{4,4}$ and $G_{6,6}$ respectively. The graph of the dodecahedron is in $G_{4,4,4,4}$. When $m=n$, graphs in $G_{m,n}$ and $\overline G_{m,n}$ are regular. 
However, besides such special cases, not much is known about graphs in $G_{m,n}$ and $\overline G_{m,n}$.

In the next section we introduce some definitions and previous results to be used in the paper. In Section~\ref{properties} we obtain results related to recognition, connectivity, Hamiltonicity, pancyclicity, bipartiteness, and diameter of graphs in $G_{m,n}$ and $\overline{G}_{m,n}$. In Section~\ref{NP} we explore the complexity of some optimization problems on these graphs. We conclude with final remarks and directions for future work in Section~\ref{concl}.    

\section{Preliminaries}
\label{prelim}

Let $G=(V,E)$ be a simple graph (i.e. with no loops or parallel edges) with vertex set $V$ and edge set $E$. The {\em neighborhood} of a vertex $v$, denoted $N(v)$, is the set of vertices adjacent to $v$; the {\em degree} of $v$, denoted $deg(v)$, is equal to $|N(v)|$, where by $|A|$ we denote the cardinality of a set $A$. The {\em degree sequence} of $G$ is the sequence of its vertex degrees. 
Graphs $G_1$ and $G_2$ are called {\em degree-equivalent}, denoted $G_1 \simeq G_2$, if they have the same degree sequence. 

Let $u,v,x,y$ be four vertices in a graph $G$ such that $uv$ and $xy$ are edges of $G$ and $ux$ and $vy$ are not edges of $G$. A \emph{2-switch} applied to $G$ is an operation that replaces the edges $uv$ and $xy$ with the edges $ux$ and $vy$. It is well known that the resulting graph has the same degree sequence as $G$, and that two graphs $G$ and $H$ have the same degree sequence if and only if there is a sequence of 2-switches that transforms $G$ into $H$ \cite{Fulkerson}. 

$G$ is {\em connected} if there is a path that connects any two vertices of $G$; otherwise $G$ is \emph{disconnected}. A {\em (connected) component} of $G$ is a maximal connected subgraph of $G$. $G$ is {\em separable} if it is disconnected or can be disconnected by removing a vertex, called a {\em cut-vertex}.  
A {\em bridge} (or {\em cut-edge}) of $G$ is an edge  whose removal  increases the number of components of $G$.
A {\em block} of $G$ is a maximal nonseparable  subgraph of $G$. If $G$ is not separable, then it is $2$-{\em connected} (or {\em biconnected}).

Given $S\subset V$, the {\em induced subgraph} $G[S]$ is the subgraph of $G$ whose vertex set is $S$ and whose edge set consists of all edges of $G$ which have both ends in $S$. The 
{\em complement} of $G$ is a  graph $\overline G$ on the same set of vertices such that two vertices of $\overline G$ are adjacent if and only if they are not adjacent in $G$. The {\em distance} $d(u,v)$ between vertices $u$ and $v$ in $G$ is the number of edges in a shortest path between $u$ and $v$ in $G$. The {\em diameter} of $G$ is defined as $diam(G) = \max_{u,v \in V} d(u,v)$.


\sloppypar A {\em clique} in graph $G$ is a complete subgraph of $G$ (i.e., a subgraph in which any two vertices are adjacent). A clique on $n$ vertices is denoted by $K_n$.
A {\em triangle} is a clique of size 3. 
The {\em clique number} of $G$, denoted $\omega(G)$, is the cardinality of a largest clique in $G$.
An {\em independent} (or {\em stable}) {\em set} of $G$ is a set of vertices no two of which are adjacent; 
the {\em independence} (or {\em stability}) {\em number} of $G$, denoted $\alpha(G)$, is the cardinality of a largest independent set in $G$. A {\em vertex cover} of $G$ is a set of vertices $S$ such that for every edge of $G$, at least one of its endpoints is in $S$.
{\sc Max Clique} and {\sc Max Independent Set} will respectively denote the optimization problems for finding a clique and independent set of maximum cardinality, and {\sc Min Vertex Cover} is the problem of finding a vertex cover of minimum cardinality. These are classical NP-hard problems \cite{GJ}.

A {\em bipartite} graph with \emph{parts} $V_1$, $V_2$, denoted $G=(V_1,V_2,E)$, is a graph where the sets $V_1$ and $V_2$ are nonempty independent sets of $G$ and partition the set of vertices of $G$. $K_{r_1,r_2}$ denotes the complete bipartite graph with parts of sizes $r_1$ and $r_2$. 

The {\em chromatic number} of $G$, denoted $\chi(G)$, is the smallest number of colors needed to color the vertices of $G$ so that no two adjacent vertices have the same color. A graph $G$ is {\em perfect} if the chromatic number of every induced subgraph of $G$ equals the clique number of that subgraph. 
A {\em co-graph} is a graph which does not contain as an induced subgraph the graph $P_4$ (a path on 4 vertices).

A {\em Hamiltonian cycle} in a graph $G$ is a cycle which visits every vertex of $G$ exactly once, and a {\em Hamiltonian path} is a path which visits every vertex of $G$ exactly once. A graph that contains a Hamiltonian cycle is called {\em Hamiltonian} and a graph that contains a Hamiltonian path is called {\em traceable}.   
A graph on $n$ vertices is called {\em pancyclic} if it contains cycles of every length from 3 to $n$. Thus, a pancyclic graph is also Hamiltonian.

For other graph theoretic definitions and notations, see \cite{BondyMurty}. Below we recall a number of theorems from the literature which  we will use in the sequel.  
\begin{theorem} (\cite{Astarian,konig})
A graph is bipartite if and only if all its cycles are of even degree.
\label{bipart}
\end{theorem}
\begin{theorem} (Dirac \cite{dirac})
A simple graph $G$ with $n \geq 3$ vertices is Hamiltonian if the degree of every vertex of $G$ is greater than or equal to $n/2$. 
\label{dirac}
\end{theorem}
\begin{theorem} (Ore \cite{ore1})
A simple graph $G$ with $n \geq 3$ vertices is Hamiltonian if for every pair of non-adjacent vertices of $G$, the sum of their degrees is greater than or equal to $n$.
\label{ore}
\end{theorem}
\begin{theorem} (Holton and Sheehan \cite{holton_sheehan}) 
If $G$ is a $2$-connected $r$-regular graph with at most $3r + 1$ vertices, 
then $G$ is Hamiltonian or $G$ is the Petersen graph.
\label{3r+1}
\end{theorem}
\begin{theorem} (Rahman and Kaykobad \cite{rahman-Kaykobad})
A simple graph with $n$ vertices has a Hamiltonian path if for every two non-adjacent vertices the sum of their degrees and the distance between them is greater than $n$.
\label{rahman-Kaykobad}
\end{theorem}
\begin{theorem} (Bondy \cite{bondy})
Any Hamiltonian graph with $n$ vertices and  at least $n^2/4$ edges is either pancyclic or is the graph $K_{n/2,n/2}$.
\label{bondy}
\end{theorem}
\begin{theorem} (Moon and Moser \cite{MoonMoser})
If the bipartite graph $G(U,V,E)$, $|U|=|V|=n$ is such that for every $k$, where $1< k< \frac{n}{2} $, the number of vertices $p \in U$ such that $deg(p) < k$ is less than $k$, and similarly with $p$ replaced by $q \in V$, then $G(U,V,E)$ is Hamiltonian.
\label{MoonMoser}
\end{theorem}
\begin{theorem} (Caro and Wei \cite{CaroWei})
A simple graph with $n$ vertices and degree sequence $d_1, d_2, \dots, d_n$ has an independence number $\alpha(G) \geq \sum_{i=1}^n \frac{1}{d_i+1}$. 
\label{caro}
\end{theorem}



\section{Properties of $G_{m,n}$ and $\overline G_{m,n}$}
\label{properties}

\subsection{Recognition}

 
Mantel and Tur\'an graphs are complete bipartite and $r$-partite graphs, respectively.
They are co-graphs and as such are perfect. 
The graphs in $G_{m,n}$ and $\overline G_{m,n}$  are neither of these, in general. 
For example, the graph from $G_{3,4}$ in Figure~\ref{fig1}, left, contains $C_5$ and $P_4$ as induced subgraphs and therefore is neither a co-graph nor perfect. The graph from the same family $G_{3,4}$ in Figure~\ref{fig1}, right, contains $P_4$ as an induced subgraph and therefore is not a co-graph, but it is perfect. These observations suggest that solving some optimization problems on graphs in $G_{m,n}$ may be hard, as we will show in Section~\ref{NP}. 

\begin{figure}
\begin{center}
\includegraphics[scale=0.45]{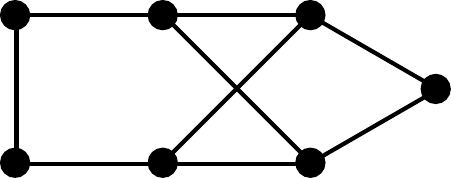}\qquad \qquad \qquad \qquad
\includegraphics[scale=0.45]{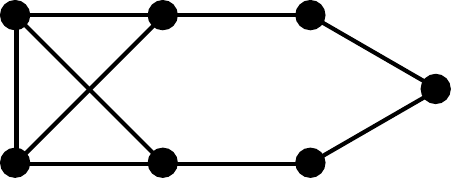}
\end{center}
\caption{Both graphs belong to the class $G_{3,4}$. {\em Left:} The graph is neither a co-graph nor perfect.
{\em Right:} The graph is not a co-graph, but it is perfect.}
\label{fig1}
\end{figure}


It is well-known that bipartite graphs (such as Mantel graphs) as well as complete $r$-partite graphs (such as Tur\'an graphs) can be recognized in polynomial time,
while recognizing incomplete $r$-partite graphs is NP-complete for $r > 2$. 
The following statement shows that $G_{p_1,\dots,p_k}$ graphs and their complements are easily recognizable, even though they could be incomplete $r$-partite for arbitrarily large $r$, and do not belong to any of the known efficiently recognizable classes of graphs.

\begin{proposition}
It can be checked in linear time whether $G \in G_{p_1,\dots,p_k}$.
\label{prop0}
\end{proposition}
\proof 
For a graph $G=(V,E)$ with $|V|=n$ and $|E|=m$, the degree sequence of $G$ can be found in $O(m)$ time, and the cardinality of each number in the degree sequence can be found in $O(n)$ time using Counting Sort. Then, it can be checked in $O(n)$ time whether the number of vertices with degree $p$ is an integer multiple of $p+1$. This is true if and only if $G\in G_{p_1,\dots,p_k}$.
\qed
\medskip

By a similar reasoning as above, it can be checked in linear time whether $G \in \overline{G}_{p_1,\dots,p_k}$.



\subsection{Connectivity, Hamitonicity, and pancyclicity}
\label{cut vertices}

\begin{figure}
\begin{center}
\includegraphics[scale=0.45]{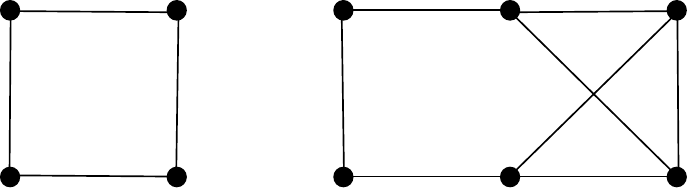}\qquad \qquad \qquad \qquad
\includegraphics[scale=0.45]{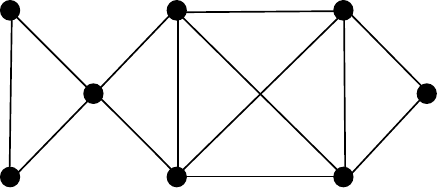}
\end{center}
\caption{{\em Left:} A graph in $G_{3,3,4}$ that is not connected.
{\em Right:} A graph in $G_{3,5}$ that is not 2-connected.}
\label{fig2}
\end{figure}

A graph that is degree-equivalent to more than two cliques does not need to be connected (see Figure~\ref{fig2}, left).   
However, the following proposition shows that a graph that is degree-equivalent to two cliques is always connected.
\begin{proposition}
Any graph $G \in G_{m,n}$ is connected. 
\label{prop00}
\end{proposition}
\proof
$G$ has $m$ vertices of degree $m-1$ and $n$ vertices of degree $n-1$. 
Suppose that $G$ has at least two components, and let $G_1=(V_1,E_1)$ be a component of minimum order.
If $|V_1| < m$, then $G_1$ has vertices of degree less than $m-1$, which is a contradiction.
If $|V_1| > m$, then all other components must be of order less than $n$.
Hence, all the degree $n-1$ vertices of $G$ must be in $G_1$.
This means that $G$ has order at least $n$, which contradicts the minimality of $G_1$. 
If $|V_1|=m$, then all vertices of $V_1$ must be of degree $m-1$, which is possible only if $G_1=K_m$. 
Let $G_2$ be the graph consisting of the other components of $G$. $G_2$ has $n$ vertices each of degree $n-1$, which is possible only if $G_2=K_n$. Then $G=K_m \cup K_n$, which contradicts the definition that $K_m \cup K_n \not\in G_{m,n}$.  
\qed

\medskip


A graph in $G_{m,n}$ does not need to be 2-connected as it may have a cut-vertex (see Figure~\ref{fig2}, right). Below we characterize all the possible numbers of cut-vertices and bridges that a graph in $G_{m,n}$ could have.

\begin{figure}
\begin{center}
\includegraphics[scale=0.42]{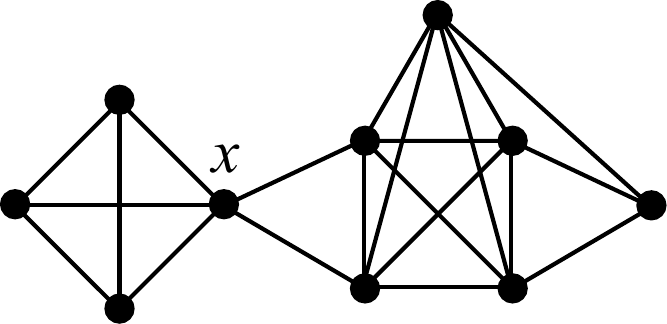}\qquad \qquad \qquad 
\includegraphics[scale=0.42]{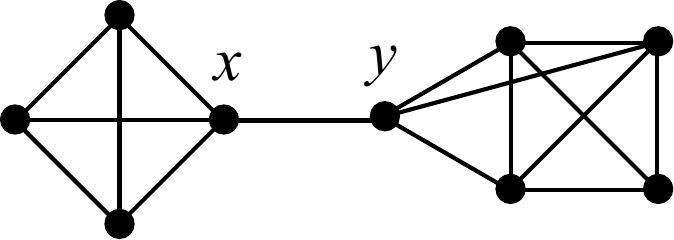}
\end{center}
\caption{{\em Left:} A graph in $G_{4,6}$ with one cut-vertex.
{\em Right:} A graph in $G_{4,5}$ with two cut-vertices and a bridge between them.}
\label{Fig33}
\end{figure}

\begin{proposition}
A graph $G \in G_{m,n}$ must have one of the following:
\begin{enumerate}
\item No cut-vertices and no bridges
\item One cut-vertex and no bridges
\item Two cut-vertices and one bridge
\item Two cut-vertices and two bridges
\item Three cut-vertices and four bridges.
\end{enumerate}
Moreover, if $G$ has at least one cut-vertex, then it is traceable.
\label{bridge_cut-vertex}
\end{proposition}
\proof 

Note that there are no graphs in $G_{m,n}$ for $m=1$ or $n=1$, nor for $m=n=2$, since any graph $G$ that is degree equivalent to $K_m\cup K_n$ for those values of $m$ and $n$ is isomorphic to $K_m\cup K_n$. The only graph in $G_{2,3}$ and $G_{3,2}$ is $P_5$, whose cut-vertices and bridges are described by Case 5 of the proposition (and $P_5$ is  traceable). The only graph $G$ in $G_{2,n}$, $n\geq 4$ or in $G_{m,2}$, $m\geq 4$, is a graph consisting of a clique of size $n$ or $m$, respectively, with one edge deleted and two leaves attached to each of the endpoints of that edge. The cut-vertices and bridges of such a graph are described by Case 4 of the proposition, and such a graph is traceable. There are also clearly graphs $G\in G_{m,n}$ that satisfy Case~1 of the proposition (for example, the ones shown in Figure \ref{fig1}). 
Moreover, the only way for a graph to have bridges and no cut-vertices is if it has at least one $K_2$ component, which by Proposition \ref{prop00} is impossible for any graph in $G_{m,n}$.

We will now show that if $G=(V,E)$ is a graph in $G_{m,n}$ with $m\geq 3$ and $n\geq 3$ and if $G$ has at least one cut-vertex, then $G$ satisfies either Case 2~or Case~3 of the proposition, and is traceable. It is well-known (see, e.g., \cite{agnarsson}) that a separable graph has at least two blocks that each contain exactly one cut-vertex. Let $G_1=(V_1,E_1)$ and $G_2=(V_2,E_2)$ be two such blocks and let $x$ be the cut-vertex of $G_1$. 

Suppose without loss of generality that $n\geq m$. If $n=m$, then all vertices in $V_1$ have degree $n-1$, and so $|V_1|\geq n$, since some non-cut vertex of $G_1$ will have at least $n-1$ neighbors in $G_1$. For the same reason, $|V_2|\geq n$. 
Then, at least one of $|V_1|$ and $|V_2|$ must equal $n$, since otherwise there will be more than $m+n$ vertices in $G$. Suppose without loss of generality that $|V_1|=n$. Then, $G_1$ is a clique of size $n$, so the cut-vertex of $G_1$ must have degree greater than $n-1$, a contradiction. Thus, $n>m$.

Suppose all vertices in $V_1$ have the same degree. Without loss of generality, suppose the degree is $n-1$. Then, $|V_1| \geq n$. Since there are $n$ total vertices in $G$ with degree $n-1$, it follows that $|V_1|=n$. Thus, $G_1$ is a clique of size $n$, so the cut-vertex of $G_1$ must have degree greater than $n-1$, which contradicts the assumption that all vertices in $V_1$ have degree $n-1$. Thus, some vertices in $V_1$ have degree $m-1$ and some have degree $n-1$. $G_1$ must have at least $m$ vertices, since some non-cut vertex of $G_1$ will have at least $m-1$ neighbors in $G_1$. If $x$ has degree $m-1$, then since $|V_1|\geq m$, it follows by a similar argument as above that $G_1$ must be a clique $K_m$, in which case $x$ cannot be a cut-vertex. Thus, $x$ must have degree $n-1$. 

Suppose first that  only $x$ has degree $n-1$, and all other vertices in $G_1$ have degree $m-1$. Then, $|V_1| \geq m$ and similarly as above it follows that $G_1$ is a clique $K_m$. If $G_2$ is the only other block of $G$, then $G$ satisfies Case 2 of the proposition. If  $G_2$ is not the only other block of $G$, then since $G_1$ only has one vertex of degree $n-1$, $G_2$ must have at least one non-cut vertex of degree $n-1$ and therefore must have $n$ total vertices. Then, since all vertices of $G$ must either be in $G_1$ or $G_2$, the only other block of $G$ must be a bridge between $G_1$ and $G_2$. This satisfies Case 3 of the proposition.

Next, suppose that at least one non-cut vertex of $G_1$ has degree $n-1$. Then, $|V_1| \geq n$, and  $|V_1| \leq n+1$ (since if $|V_1| > n+1$ then $|V_2| < m$, and so the vertices in $V_2$ could not have degree at least $m-1$). If $|V_1|=n+1$, then $|V_2|=m$ with $V_1$ and $V_2$ sharing a vertex, and so $G_1$ and $G_2$ must be the only blocks of $G$. This satisfies Case 2 of the proposition. If $|V_1|=n$ and $|V_2|=m$, then  $G_1$ and $G_2$ are two blocks separated by a bridge, which satisfies Case 3 of the proposition. If $|V_1|=n$ and $|V_2|=m+1$, then $G_1$ and $G_2$ are  the only blocks of $G$, and $G$ satisfies Case 2 of the proposition.

Finally, for all types of separable graphs described above, the blocks $G_1$ and $G_2$ are Hamiltonian by Theorem~\ref{ore}. Thus, $G$ is traceable since a Hamiltonian path of $G_1$ ending at the cut vertex of $G_1$ can be combined with a Hamiltonian path of $G_2$ ending at the cut vertex of $G_2$, through the cut-vertex or bridge between the two blocks. 
\qed

\medskip
Proposition \ref{bridge_cut-vertex} shows that a graph in $G_{m,n}$ with a cut-vertex must be traceable (and cannot be Hamiltonian).
In $G_{m,n}$ there are also graphs without a cut-vertex that are traceable but not Hamiltonian, as well as graphs without a cut-vertex that are Hamiltonian.
Figure~\ref{fig44} shows two such graphs. In fact, we make the following conjecture about the traceability of graphs in $G_{m,n}$.

\begin{conjecture}
Every graph $G \in G_{m,n}$ is traceable.   
\label{conjecture}
\end{conjecture}

In the remainder of this section we provide various evidence for this conjecture. We begin with a more general construction of infinite families of 2-connected Hamiltonian graphs in $G_{m,n}$. A \emph{twin bridge} joining graphs $G_1$ and $G_2$ consists of two edges which connect two vertices of $G_1$ with two vertices of $G_2$.

\begin{figure}
\begin{center}
\includegraphics[scale=0.40]{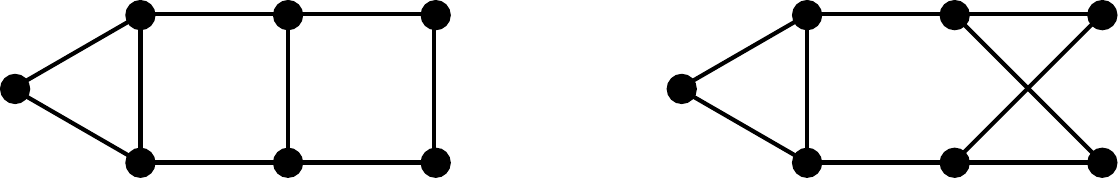}
\end{center}
\caption{Illustration to the proof of Proposition~\ref{prop_twicebridge} showing two 2-connected graphs from $G_{3,4}$. 
{\em Left:} The graph is Hamiltonian. {\em Right:} The graph is traceable but not Hamiltonian.}
\label{fig44}
\end{figure}


\begin{proposition}
Let $G$ be a graph in $G_{m,n}$, $n > m \geq 3$ which can be obtained by joining two 2-connected graphs by a twin bridge. Then $G$ is Hamiltonian except for the graph in Figure~\ref{fig44}, right, which is traceable. 
\label{prop_twicebridge}
\end{proposition}
\proof The proof is similar to the proof of Proposition~\ref{bridge_cut-vertex}, therefore some details are omitted. By arguments that parallel those used in the proof of Proposition~\ref{bridge_cut-vertex} it can be shown that $G$ must have one of the following two structures.

{\em Case 1:} \ $G$ can be obtained by removing an edge $e=uv$ from a clique $K_m$  and an edge $e'=u'v'$ from a clique $K_n$ and adding the edges $uu'$ and $vv'$. See Figure~\ref{mn}, left for an example. We will refer to $K_m-e$ as $G_1$ and $K_n-e'$ as $G_2$.

{\em Case 2:} $G$ can be obtained by removing two non-incident edges $e=uv$ and $e'=u'v'$ from a clique $K_n$ and joining two vertices among $\{u,v,u',v'\}$ with two vertices $a$ and $b$ of a clique $K_m$ where $m=n-1$ by two edges. See Figure~\ref{mn}, right for an example. We will refer to $K_m$ as $G_1$ and $K_n-e-e'$ as $G_2$.

We will first consider the situation when $m \geq 4$ (and hence $n \geq 5$). Then, in both cases, by Theorem~\ref{dirac} both $G_1$ and $G_2$ are Hamiltonian. If $G$ satisfies Case~1, then a Hamiltonian cycle of $G$ can be obtained by merging a Hamiltonian path of $G_1$ ending at $u$ and $v$ with a Hamiltonian path of $G_2$ ending at $u'$ and $v'$ using the edges $uu'$ and $vv'$. 
If $G$ satisfies Case 2, let $x_1,\ldots,x_k$ be the vertices of $G_2$ different from $u$, $v$, $u'$, and $v'$. Note that since $n\geq 5$, $k$ is at least 1. If the two vertices among $\{u,v,u',v'\}$ which are connected to $a$ and $b$ have an edge between them, suppose without loss of generality that they are $u$ and $u'$. Then, a Hamiltonian cycle of $G$ can be obtained by merging a Hamiltonian path of $G_1$ ending at $a$ and $b$ with the Hamiltonian path $u, v', x_1,\ldots,x_k, v, u'$ of $G_2$ using the two edges between $G_1$ and $G_2$. 
If the two vertices among $\{u,v,u',v'\}$ which are connected to $a$ and $b$ do not have an edge between them, suppose without loss of generality that they are $u$ and $v$. Then, a Hamiltonian cycle of $G$ can be obtained by merging a Hamiltonian path of $G_1$ ending at $a$ and $b$ with the Hamiltonian path $u, u', x_1,\ldots,x_k, v', v$ of $G_2$ using the two edges between $G_1$ and $G_2$. 

Finally, consider the case when $m=3$. Then the structure of $G$ must satisfy Case 2, because in Case 1, $G_1$ would not be 2-connected. Then, since $m=n-1$, $G$ is a graph in $G_{3,4}$. There are only two non-isomorphic graphs in $G_{3,4}$ that satisfy Case 2, and they are depicted in Figure~\ref{fig44}. One of them is Hamiltomian; the other one is traceable and not Hamiltonian, and it is the unique graph with this property.      
\qed

\begin{figure}
\begin{center}
\includegraphics[scale=0.48]{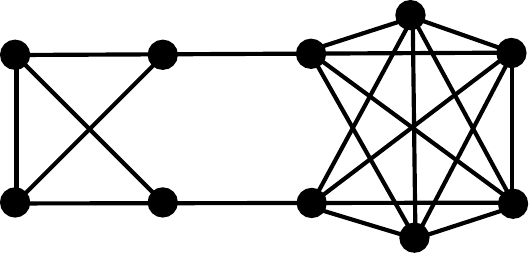}\qquad \qquad \qquad
\includegraphics[scale=0.48]{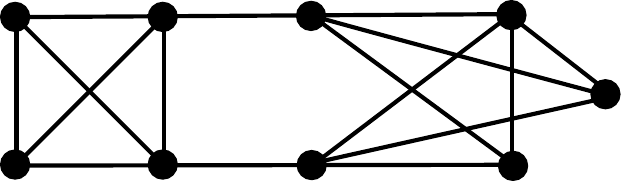}
\end{center}
\caption{Illustration to the proof of Proposition~\ref{prop_twicebridge}: Case 1 ({\em Left}) and Case 2 ({\em Right}).}
\label{mn}
\end{figure}


\medskip
For the special case of Conjecture \ref{conjecture} where $m=n$ we have the following characterization. 
\begin{proposition}
Any graph $G \in G_{n,n}$ is 2-connected and Hamiltonian.
\label{lemma1}
\end{proposition}
\proof In the proof of Proposition~\ref{bridge_cut-vertex}, it was shown that when $m,n\geq 3$ and $n=m$, 
a graph in $G_{m,n}$ cannot have a cut-vertex. It was also shown that $G_{1,1}$ and $G_{2,2}$ are empty. Thus, any graph $G\in G_{n,n}$ is 2-connected. Next, note that the Petersen graph is not in $G_{n,n}$ for any $n$, since if it was, it would have to belong to $G_{5,5}$ as it has 10 vertices, but the Petersen graph is 3-regular while the graphs in $G_{5,5}$ are 4-regular. 
Thus, since a graph $G\in G_{n,n}$ is a 2-connected $(n-1)$-regular graph on $2n$ vertices different from the Petersen graph, and $2n \leq 3(n-1)+1=3n-2$ for any $n \geq 2$, by Theorem~\ref{3r+1} it follows that $G$ is Hamiltonian.
\qed
\medskip
We now explore the connectivity and Hamiltonicity of graphs in $\overline G_{m,n}$. We first show that all graphs in $\overline G_{m,n}$ with $m\neq n$ are 2-connected. We then show that when $m=n$, the graphs in $\overline G_{m,n}$ are not only 2-connected but also pancyclic. 
\begin{proposition}
Any graph $G \in \overline G_{m,n}$ with $m\neq n$ is 2-connected.
\label{complement Gmn 2-connected}
\end{proposition}
\proof
Suppose without loss of generality that $m<n$. Note that there are no graphs in $\overline{G}_{m,n}$ for $m=1$ or $n=1$, nor for $m=n=2$, since any graph $G$ that is degree-equivalent to $K_m\cup K_n$ for those values of $m$ and $n$ is isomorphic to $K_m\cup K_n$, and therefore $G_{m,n}$ and $\overline{G}_{m,n}$ are empty for those values of $m$ and $n$.

If $G \in \overline G_{m,n}$, then $G$ has $m$ vertices of degree $n$ and $n$ vertices of degree $m$. We will first show that $G$ is connected. Suppose for contradiction that this is not the case, and let $G_1$ be an arbitrary component of $G$. If $G_1$ has only vertices of degree $n$, then $G_1$ must have at least $n+1>m$ vertices, which contradicts the fact that there are $m$ vertices of degree $n$. 
If $G_1$ has only vertices of degree $m$, then $G_1$ has at least $m+1$ vertices. Then the other components of $G$ would collectively have fewer than $n$ vertices, so those vertices could not have degree $n$. Thus, $G$ could not have any degree $n$ vertices, a contradiction. 
If $G_1$ has vertices of degree $m$ and $n$, then $G_1$ must have at least 
$n+1$ vertices. Then the other components of $G$ would collectively  have fewer than $m$ vertices, so they could not have degree $m$ nor degree $n$, a contradiction.

We will now show that $G$ is 2-connected. Suppose for contradiction that this is not the case, and let $G_1$ and $G_2$ be two blocks of $G$, each of which contains exactly one cut-vertex. If all vertices of $G_1$ (possibly except the cut vertex) have  degree $n$, then $G_1$ must have at least $n+1$ vertices; since $G$ has a total of $m<n$ vertices of degree $n$, this is a contradiction. If all vertices of $G_1$ (possibly except the cut vertex) have degree $m$, then $G_1$ must have at least $m+1$ vertices. Then the other blocks of $G$ would collectively have at most $n$ vertices, so they could not have any vertices of degree $n$. Then $G$ would not have any vertices of degree $n$ (possibly except the cut vertex of $G_1$), a contradiction. 
The above can be repeated identically with $G_2$ in place of $G_1$. Thus, it follows that both $G_1$ and $G_2$ must have both degree $m$ and degree $n$ vertices. 

Suppose at least one of $G_1$ and $G_2$, say $G_1$, has a vertex of degree $n$ which is not a cut-vertex. Then $G_1$ has at least $n+1$ vertices and $G_2$ has at most $m$ vertices. 
Then the non-cut vertices of $G_2$ must have degree less than $m$, which is a contradiction. Thus, it follows that the only vertices of degree $n$ of $G_1$ and $G_2$ are their cut-vertices. 
Then each of $G_1$ and $G_2$ as an induced subgraph must have at least $m+1$ vertices. We will consider two possibilities.


{\em Case 1:} $G_1$ and $G_2$ share a cut vertex $x$; there may or may not be other blocks which also share that cut vertex. Then, the graph obtained  by removing $G_1$ and $G_2$  from $G$
would have at most $m+n - (m+1) - m = n-m-1 < n$ vertices. Thus, $G$ cannot have vertices of degree $n$, possibly except the vertex $x$. This contradicts the assumption that there are $m\geq 2$ vertices of degree $n$. 

\begin{figure}[ht]
    \centering
    \includegraphics[scale=.45]{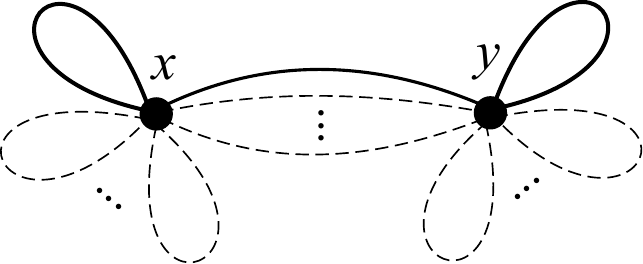}
    \caption{Structure of graphs with two cut vertices of degree $n>2$ and all other vertices having degree 2. Curves between $x$ and $y$ represent paths of length at least 1, ovals attached to $x$ and $y$ represent cycle blocks of size at least 3. Dashes mean the structure is possibly nonexistent.}
    \label{fig_m2}
\end{figure}

{\em Case 2:} $G_1$ and $G_2$ do not share a cut vertex; then both $G_1$ and $G_2$ share their respective cut vertex, say $x$ and $y$, with at least one other block. Then, the graph obtained  by removing $G_1$ and $G_2$  from $G$
would have at most $m+n - (m+1) - (m+1) = n-m-2 < n$ vertices. Thus, $G$ cannot have vertices of degree $n$, possibly except the vertices $x$ and $y$. If $m>2$, this contradicts the assumption that there are $m$ vertices of degree $n$. Thus, suppose $m=2$. Then $x$ and $y$ have degree $n$ and there are $n$ other vertices with degree 2. It follows that $G$ is a graph consisting of the two vertices $x$ and $y$, one or more paths of length at least 1 between $x$ and $y$, and one or more cycle blocks attached to each of $x$ and $y$; see Figure \ref{fig_m2} for an illustration. It is easy to see that for a graph with this structure, the number of vertices of degree 2 exceeds the degree of $x$ and $y$ by at least 1. Thus, there are more than $n$ vertices of degree 2, which is a contradiction.
\qed
\medskip

\begin{proposition}
Every graph $G \in \overline G_{n,n}$ is pancyclic.  
\label{complement Gnn pancyclic}
\end{proposition}
\proof
A graph $G=(V,E)\in G_{n,n}$ has $2n$ vertices, all of degree $n$. Thus, $G$ satisfies the conditions of Dirac's and Ore's theorems (Theorem~\ref{dirac} and Theorem~\ref{ore}) and hence is Hamiltonian. Moreover,
$|E|=\frac{1}{2}\sum_{v\in V}deg(v)=\frac{1}{2}(2n)n=n^2=\frac{|V|^2}{4}$. Since graphs in $\overline G_{n,n}$ are complements of graphs in $G_{n,n}$, and graphs in $G_{n,n}$ are different from a disjoint union $K_n \cup K_n$, it follows that $G\neq K_{n,n}$. Then, Bondy's theorem (Theorem~\ref{bondy}) implies that $G$ is pancyclic. 
\qed


\medskip
Finally, we consider graphs in $\overline G_{m,n}$ for the special case when $m=n-1$ and show that such graphs are traceable.

\begin{proposition}
Every graph $G \in \overline G_{n,n-1}$ is traceable.  
\label{complement Gnn-1}
\end{proposition}
\proof
A graph $G\in \overline G_{n,n-1}$ has $2n-1$ vertices,  $n$ of which have degree $n-1$ and $n-1$ of which have degree $n$. Then, for any two non-adjacent vertices 
$u$ and $v$, $deg(u)+deg(v)+d(u,v) \geq (n-1)+(n-1)+2 = 2n$, and $2n$ is greater than the number of vertices of $G$. Thus, by Theorem~\ref{rahman-Kaykobad}, $G$ is traceable.  
\qed


\subsection{Bipartiteness}

In this section we characterize the graphs from $G_{m,n}$ and $\overline G_{m,n}$ that are bipartite. We also show that bipartiteness of a graph in $G_{m,n}$ implies traceability, which provides more evidence for Conjecture \ref{conjecture}.

\begin{theorem}
If a graph $G\in G_{m,n}$ is  bipartite with parts $V_1$ and $V_2$, $|V_1|\leq |V_2|$,  
then one of the following conditions holds:
\begin{enumerate}
\item[\text{1)}] $|V_1|=|V_2|$ and $n=m+2$ and $m,n$ are even
\item[\text{2)}] $|V_1|=m$ and $|V_2|=n$ and $n=m+1$
\item[\text{3)}] $|V_1|=|V_2|$ and $m=n$.
\end{enumerate}
Moreover, in the first case $G$ is Hamiltonian, in the second case $G$ is traceable, and in the third case $G$ is Hamiltonian.   
\label{bipartite}
\end{theorem}

\begin{proof}
Let $G=(V_1,V_2,E)\in G_{m,n}$ be a bipartite graph. If $n=m$, then $G$ is regular and must therefore be balanced. Thus, in this case, Condition 3) is satisfied. Moreover, by 
Proposition \ref{lemma1}, $G$ is Hamiltonian. 
If $m<3$ or $n<3$, then $G$ is $P_5$, in which case Condition 2) is satisfied. Note that any other graph in $G_{m,n}$ with $m<3$ or $n<3$ is a graph consisting of a clique of size at least 4 with one edge deleted and two leaves attached to each of the endpoints of that edge, in which case $G$ is not bipartite. Thus, suppose hereafter that $n>m\geq 3$. 

Suppose $|V_1|=|V_2|$. Then the number of degree $m-1$ vertices in both $V_1$ and $V_2$ must be $\frac{m}{2}$, and the number of degree $n-1$ vertices in both $V_1$ and $V_2$ must be $\frac{n}{2}$, since otherwise the part with the larger number of degree $n-1$ vertices will have a higher total degree than the other part. Then, $m$ and $n$ must both be even. Moreover, since both parts contain a degree $n-1$ vertex, and both parts contain $\frac{m}{2}+\frac{n}{2}$ vertices, it follows that $\frac{m}{2}+\frac{n}{2}\geq n-1$. Thus, $m\geq n-2$, but also by assumption $m<n$, so $n-1\geq m\geq n-2$. Finally, $m$ cannot equal $n-1$, since both $m$ and $n$ are even, so $m=n-2$. Thus, in this case Condition 1) is satisfied.

Next, let $G=(V_1,V_2,E) \in G_{m,n}$, $n>m\geq 3$, and 
suppose $|V_1|\neq |V_2|$. 
Note that both parts of $G$ must contain some vertices of degree $m-1$ and some vertices of degree $n-1$, since otherwise the part with all the degree $n-1$ vertices will have a higher total degree than the other part. Thus, $|V_2|>|V_1|\geq n-1> m-1$, so $|V_1|\geq m$ and $|V_2|\geq m+1$.
Let $k=n-m$. Since $k$ is a nonnegative integer, $k=\lceil k/2\rceil +\lfloor k/2\rfloor$. Thus, $|V_1|+|V_2|=m+n=2m+k=2m+\lceil k/2\rceil +\lfloor k/2\rfloor$. 
Note that $|V_1|\leq m+\lfloor k/2\rfloor$, since otherwise $|V_2|$ would not be bigger than $|V_1|$.
Thus, we can let $|V_1|=m+\lfloor k/2 \rfloor - r$ for some $r\geq 0$, which means  $|V_2|=m+\lceil k/2 \rceil + r$.
Since $V_2$ contains vertices of degree $n-1$, it follows that $|V_1|=m+\lfloor k/2 \rfloor - r \geq
n-1=m+k-1$. Thus, 
$\lfloor k/2 \rfloor - r \geq
k-1$, which implies 
$k=2, r=0$ or $k=1, r=0$.
In the former case  $|V_1|=m+1=|V_2|$, which contradicts the assumption that $|V_1|\neq |V_2|$. 
In the latter case, $|V_1|=m$ and $|V_2|=m+1$, and since $k=1$, $n=m+1$. Thus, in this case Condition 2) is satisfied.

Now suppose that $G=(V_1,V_2,E)\in G_{m,n}$, $n>m\geq 3$, is a bipartite graph with $|V_1|=|V_2|$ and $n=m+2$. We will show that $G$ is Hamiltonian.
Since $n=m+2$, it follows that \[n-1 \geq \frac{n-1}{2}=
\frac{n+n-2}{4}=
\frac{n+m}{4}=\frac{|V_1|+|V_2|}{4}=\frac{2|V_1|}{4}=\frac{|V_1|}{2}.\]
Moreover, since 
$m\geq 3$, it follows that 
\[m-1\geq\frac{m+1}{2}=\frac{m+m+2}{4}=\frac{m+n}{4}=
\frac{|V_1|+|V_2|}{4}=\frac{2|V_1|}{4}=\frac{|V_1|}{2}.\]
Since the only degrees of vertices in $G$ are $n-1$ and $m-1$, it follows that there are no vertices in $G$ of degree less than $\frac{|V_1|}{2}$. Thus, for every $k$ where $1< k< \frac{|V_1|}{2}$, the number of vertices $v \in V_1$ with $deg(v) < k$ is 0 and therefore is less than $k$. The same holds if $V_1$ is replaced by $V_2$. Thus, by Theorem~\ref{MoonMoser}, $G$ is Hamiltonian.

Finally, suppose that $G=(V_1,V_2,E)\in G_{m,n}$, $n>m\geq 3$, is a bipartite graph with $|V_1|=m$ and $|V_2|=n$ and $n=m+1$. We will show that $G$ is traceable. Let $V_2$ have $a$ vertices of degree $m$ and $b$ vertices of degree $m-1$, and let $V_1$ have $c$ vertices of degree $m$ and $d$ vertices of degree $m-1$. Since $G$ is bipartite, we have 
\begin{equation}
am+b(m-1)=cm+d(m-1).
\label{eq1}
\end{equation} 
Since $|V_1|=m$ and $|V_2|=n=m+1$, we also have $a=m+1-b$ and $c=m-d$. Substituting $a$ and $c$ in (\ref{eq1}) we obtain $b=m+d$. 
On the other hand, the number of degree $m-1$ vertices in $G$ is $b+d=m$. Thus, it follows that $d=0$, and hence $b=m$, $a=1$, and $c=m$.
This means all $m$ vertices in $V_1$ have degree $m$, while $V_2$ has $m$ vertices of degree $m-1$ and one vertex of degree $m$. 
Let $v$ be the vertex of degree $m$ in $V_2$, and let 
$G'=G-v$. $G'$ is a balanced $(m-1)$-regular bipartite graph.
Since $m\geq 3$, there are no vertices in either part of $G'$ with degree less than or equal to $k$ for $1< k< \frac{m}{2}$, so by Theorem~\ref{MoonMoser}, $G'$ is Hamiltonian. Then $G$ is traceable.  
\qed
\end{proof}

\begin{proposition}
Let $G\in \overline G_{m,n}$. Then, $G$ is not bipartite.
\label{prop_bipart}
\end{proposition}
\proof 
Suppose first that $m=n$. Then, the complement of $G$, $\overline G$, is a graph on $2n$ vertices each with degree $n-1$, and $\overline G$ is different from the disjoint union $K_n \cup K_n$. By Caro-Wei's theorem (Theorem~\ref{caro}), given a graph $H$ with a degree sequence $d_1,d_2,\dots,d_r$, $\alpha(H) \geq \sum_{i=1}^r \frac{1}{d_i+1}$, with equality holding only when $H$ is a disjoint union of $r$ cliques, in which case $\alpha(H)=r$. Applying Caro-Wei's theorem to $\overline G$, we have $\alpha(\overline G)\geq \sum_{i=1}^{2n}\frac{1}{n}=2$; however, since $\overline G$ is not a disjoint union of two cliques, it follows that $\alpha(\overline G)\neq 2$, so $\alpha(\overline G) \geq 3$. Hence,  $\omega(G) \geq 3$, which means that $G$ contains a triangle and by Theorem \ref{bipart} it is not bipartite.

Now suppose that $m\neq n$, and without loss of generality, suppose $n>m$.  Then $G$ has $n$ vertices of degree $m$ and $m$ vertices of degree $n$. Suppose for contradiction that $G$ is bipartite with parts $V_1$ and $V_2$. Let $V_1$ have a vertex of degree $n$. Then $|V_2| \geq n$ and $|V_1| \leq m <n$. Then $V_2$ has no vertices of degree $n$, i.e., all vertices of $V_2$ are of degree $m$. Hence $G=K_{m,n}$, but this contradicts the assumption that $G$ is the complement of a graph which is different from the disjoint union of $K_m$ and $K_n$.
\qed

\subsection{Diameter}
\label{diameter}

In this section we show that the diameters of graphs in $G_{m,n}$ and $\overline G_{m,n}$ are bounded by small constants. 

\begin{proposition}
If $G \in G_{m,n}$, then $diam(G) \leq 4$. 
\label{diam(Gmn)}
\end{proposition}
\proof
Let $u$ and $v$ be vertices of $G=(V,E)\in G_{m,n}$, and suppose without loss of generality that $n\geq m$. We will show that $d(u,v)\leq 4$. If $u$ and $v$ are adjacent, then $d(u,v)=1$. Otherwise, we consider two cases.

{\em Case 1:} \ [$m=n$] or [$n>m$ and at least one of  $u$ and $v$ has degree $n-1$]. If 
$N(u)\cap N(v)\neq \emptyset$, then $d(u,v)=2$. If $N(u)\cap N(v)= \emptyset$, then $N(u) \cup \{u\} \cup N(v) \cup \{v\} = V$, and since  $G$ is connected, there must be adjacent vertices  $x \in N(u)$ and $y \in N(v)$. Then, $d(u,v)=3$.

{\em Case 2:} \ [$n>m$ and both $u$ and $v$ have degree $m-1$].  
If $N(u)\cap N(v)\neq \emptyset$, then $d(u,v)=2$. If $N(u)\cap N(v)= \emptyset$, then since $u$ and $v$ are not adjacent, each of them has no more than $m-2$ neighbors of degree $m-1$. 
Hence, both $u$ and $v$ have at least one neighbor of degree $n-1$. Let $x$ and $y$ be such neighbors of $u$ and $v$, respectively. If $x$ and $y$ are adjacent, then $d(u,v)=3$. 
If they are not adjacent, then since $n>m$, $|N(x)|+ |N(y)|=2n-2>m+n-2=|V\backslash \{x,y\}|\geq|N(x)\cup N(y)|$, so $|N(x)\cap N(y)|\neq \emptyset$.
Then, since $x$ and $y$ have a common neighbor, $d(u,v) = 4$.
\qed

\begin{proposition}
If $G \in \overline G_{m,n}$, then $diam(G) \leq 4$. 
\label{diam(complement(Gmn))}
\end{proposition}
\proof
Let $u$ and $v$ be vertices of $G=(V,E)\in \overline{G}_{m,n}$, and suppose without loss of generality that $n\geq m$. We will show that $d(u,v)\leq 4$. If $u$ and $v$ are adjacent, then $d(u,v)=1$. Otherwise, we consider two cases.

{\em Case 1:} \ [$m=n$] or [$n>m$ and at least one of  $u$ and $v$ has degree $n$]. 
Since $|N(u)|+ |N(v)|\geq m+n>|V\backslash \{u,v\}|\geq|N(u)\cup N(v)|$, it follows that $|N(u)\cap N(v)|\neq \emptyset$. Thus, $u$ and $v$ have a common neighbor, so $d(u,v)=2$. 

{\em Case 2:} \ [$n>m$ and both $u$ and $v$ have degree $m$].
Let $x$ be any vertex of degree $n$. Then, since there are a total of $m+n$ vertices in $G$ and since $G$ is connected, $x$ must be adjacent to a vertex from $N(u) \cup \{ u \}$ and to a vertex  from $N(v) \cup \{ v \}$. Thus, $d(u,v)\leq 4$.
\qed
\medskip

\begin{figure}[ht]
    \centering
    \includegraphics[scale=.3]{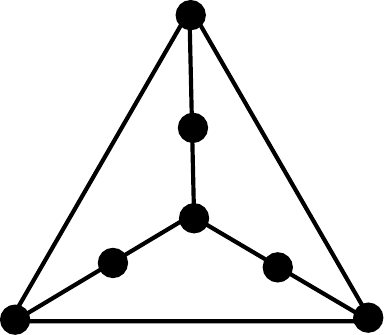}
    \caption{A graph in $G_{3,4}$ with diameter 2.}
    \label{fig_diam2}
\end{figure}
Note that there are graphs in $G_{m,n}$ with diameters 2, 3, and 4 (for example, those shown in Figures \ref{fig_diam2}, \ref{fig1}, and \ref{Fig33}, respectively). However, we have not found any graphs in $\overline{G}_{m,n}$ with diameter 4. We leave it as an open question to determine whether there are such graphs,  or whether $diam(G) \leq 3$ for all $G \in \overline G_{m,n}$.

\section{NP-hardness of problems on $G_{m,n}$ and $\overline G_{m,n}$}
\label{NP}
The properties discussed in the previous section are beneficial regarding possible practical applications as they make certain optimization problems easier to solve. On the other hand, some classical optimization problems are hard to solve on graphs in $G_{m,n}$ and $\overline G_{m,n}$. The next theorem provides an example of this.   
\begin{theorem}
{\sc Max Independent Set} is NP-hard on $G_{n,n}$.
\end{theorem}
\proof
Let $G=(V,E)$ with $|V|=n$ be a simple connected graph. We construct a graph $G'=(V',E')$ as follows. 
\begin{enumerate}
\item $V' = V_1' \cup V_2'$, 
where $V_1'=\{(x,1):x \in V \}$, $V_2'=\{(x,2):x \in V \}$. 
\item
$(x,1)$ is adjacent to $(y,1)$ and $(x,2)$ is adjacent to $(y,2)$ in $G'$ if and only if 
$x$ is adjacent to $y$ in $G$. 
\item
$(x,1)$ is adjacent to $(y,2)$ and $(x,2)$ is adjacent to $(y,1)$ in $G'$ if and only if 
$x$ is not adjacent to $y$ in $G$ for all $x,y$, $x \neq y$. 
\end{enumerate}
$G'$ consists of two copies of $G$ with  some vertices of the first copy connected to some vertices of the second copy so that all vertices have degree $n-1$. $G'$ can be constructed in  polynomial time, and $G'\in G_{n,n}$. 

Let $I$ be a maximum independent set of $G'$. The vertices of $I$ can be partitioned into $I_1$ and $I_2$, where $I_1=I \cap V_1'$ and $I_2=I\cap V_2'$. Let
\[
S=\begin{cases}
\{x\in V:(x,1)\in I_1\}, \text{ if } |I_1|\geq |I_2|\\
\{x\in V:(x,2)\in I_2\}, \text{ if } |I_2|\geq |I_1|.
\end{cases}
\]
Let $J$ be a maximum independent set of $G$. We will now show that $|S|$ is a 2-approximation for $|J|$.
Let $J_1=\{(x,1):x\in J\}$ be the copy of $J$ contained in $V_1'$.
Then, since $J_1$ is an independent set in $G'$ and $I$ is a maximum independent set in $G'$, we  have $|J|=|J_1|\leq |I|=|I_1|+|I_2|$. If $|I_1| \geq |I_2|$, then $|J| \leq 2|I_1|=2|S|$. Similarly, if $|I_1| \leq |I_2|$, then $|J|\leq 2|I_2|=2|S|$. In either case,  $|S|$ is a 2-approximation for $|J|$.
Since the maximum independent set problem cannot be approximated to a constant factor in polynomial time unless $P = NP$ \cite{hastad}, it follows that {\sc Max Independent Set} is NP-hard on $G_{n,n}$.
\qed

\medskip

Given a simple graph $G$, it is well-known that $Q$ is a clique in $G$, if and only if $Q$ is an independent set in $\overline G$, if and only if $V-Q$ is a vertex cover in $\overline G$. Thus, we have the following corollary.

\begin{corollary}
The {\sc Max Clique} and {\sc Max Vertex Cover} problems are NP-hard on $\overline G_{n,n}$.
\end{corollary}

\section{Conclusion}
\label{concl}

In this paper we studied various properties of  graphs that are degree equivalent to complete bipartite graphs or to disjoint unions of cliques. We showed that such graphs are connected and have a bounded diameter, and are therefore small-world networks. We characterized when these graphs are biconnected and bipartite, and showed that many of them have desirable qualities such as traceability, Hamiltonicity, and even pancyclicity. We also showed that some optimization problems are hard on these graphs. Below are several open questions and directions for future research.
\begin{enumerate}
\item[1)] Is every graph in $G_{m,n}$ traceable?
\item[2)] Is there a graph in $\overline{G}_{m,n}$ with diameter 4?
\item[3)] Can the obtained results be extended to multigraphs, and graphs that are degree-equivalent to more than two cliques or to complete multipartite graphs? 
\end{enumerate}

Regarding question 1), we have shown that all non-biconnected graphs and all bipartite graphs in $G_{m,n}$ are traceable. Moreover, all graphs in $G_{n,n}$ are traceable, as are all graphs in $G_{m,n}$ that can be obtained by joining two biconnected graphs by a twin bridge. Thus, we believe the answer to this question may be positive. Regarding question 2), we have searched many randomly generated graphs from $\overline{G}_{m,n}$ with a computer and have not found such a graph. Thus, we believe the answer to this question may be negative.
Regarding question~3), we note that many of the properties that hold for simple graphs in $G_{m,n}$ do not hold for multigraphs or for graphs in $G_{p_1,\dots,p_k}$ for $k>2$. For example, multigraphs that are degree-equivalent to two disjoint cliques are not always connected, and therefore are not always Hamiltonian or traceable. The same holds for simple graphs that are degree equivalent to three or more disjoint cliques. However, it would be interesting to obtain sufficient conditions that guarantee these properties.


\end{document}